\documentclass[reqno]{amsart}
\usepackage{amsmath,amsthm,amscd,amssymb,latexsym,upref,stmaryrd,pdfsync,amsfonts,cite}
\usepackage[colorlinks=true,naturalnames=true,plainpages=false,pdfpagelabels=true]{hyperref}

\allowdisplaybreaks \numberwithin{equation}{section}

\makeatletter

\newcommand{\Romnumb}[1]{\expandafter\@slowromancap\romannumeral #1@}
\makeatother

\newcommand*{\mailto}[1]{\href{mailto:#1}{\nolinkurl{#1}}}

\newtheorem{theorem}{Theorem}[section]
\newtheorem{definition}[theorem]{Definition}
\newtheorem{lemma}[theorem]{Lemma}
\newtheorem{corollary}[theorem]{Corollary}
\newtheorem{remark}[theorem]{Remark}

\newcommand{\R}{{\mathbb R}}
\newcommand{\Z}{{\mathbb Z}}
\newcommand{\be}{\begin{equation}}
\newcommand{\ee}{\end{equation}}
\newcommand{\ti}{\tilde}
\newcommand{\floor}[1]{\lfloor#1 \rfloor}
\newcommand{\ceil}[1]{\lceil#1 \rceil}
\newcommand{\abs}[1]{\lvert#1\rvert}
\newcommand{\norm}[1]{\left\Vert#1\right\Vert}
\newcommand{\up}{(+1)}
\newcommand{\down}{(-1)}
\newcommand{\nc}{(0)}
\newcommand{\sig}{\sigma}
\newcommand{\lam}{\lambda}
\newcommand{\RM}[1]{\MakeUppercase{\romannumeral #1}}

\renewcommand{\ge}{\geqslant}
\renewcommand{\le}{\leqslant}

\numberwithin{equation}{section}

\begin{document}

\title{Relative Oscillation Theory for Jacobi Matrices Extended}

\author[K. Ammann]{Kerstin Ammann}
\address{Faculty of Mathematics \\ University of Vienna \\ Nordbergstrasse 15\\ 1090 Wien\\ Austria}
\email{\mailto{Kerstin.Ammann@univie.ac.at}}
\urladdr{\url{http://www.mat.univie.ac.at/~kerstin/}}

\thanks{{\it Research supported by the Austrian Science Fund (FWF) under Grant No.~Y330}}


\begin{abstract}
We present a comprehensive treatment of relative oscillation theory for finite Jacobi matrices. We show that the difference of the number of eigenvalues of two Jacobi matrices in an interval equals the number of weighted sign-changes of the Wronskian of suitable solutions of the two underlying difference equations. Until now only the case of perturbations of the main diagonal was known. We extend the known results to arbitrary perturbations, allow any (half-)open and closed spectral intervals, simplify the proof, and establish the comparison theorem.
\end{abstract}

\maketitle

\section{Introduction}

Jacobi operators appear at numerous occasions in mathematics as well as in physical models. For example, they are intimately related to the theory of orthogonal polynomials or constitute as a simple one-band tight binding model in quantum mechanics. They can be viewed as the discrete analog of Sturm--Liouville operators and their investigation has many similarities with Sturm--Liouville theory. Moreover, spectral and inverse spectral theory for Jacobi operators plays a fundamental role in the investigation of the Toda lattice and its modified counterpart, the Kac--van Moerbeke lattice. For further information we refer, e.g., to \cite{jacop}.

Let $a,b \in \ell(\Z)=\{\varphi~|~\varphi:\Z\to\R\}$, where $a(n)<0$ holds for all $n\in\Z$. Then, the \emph{Jacobi matrix}
\be J=\begin{pmatrix}
 b(1) & a(1) & 0 & 0 & 0 \\
 a(1) & b(2) & \ddots & 0 & 0 \\
0 & & \ddots & & 0 \\
0 & 0 &	\ddots & b(N-2) & a(N-2) \\
0 & 0 & 0 & a(N-2) & b(N-1) \\
\end{pmatrix} \ee
is self-adjoint and $\sig(J)$ is real and simple. The corresponding \emph{Jacobi difference equation} \be\label{eq:jde} \tau u = z u, \ee
where
$$\begin{array}{lllcl}
\tau: &\ell(\Z) &\rightarrow \ell(\Z) && \\
&u(n) &\mapsto (\tau u)(n) &=& a(n) u(n+1) + a(n-1) u(n-1) + b(n) u(n) \\
&		 &										&=& \partial (a(n-1) \partial u(n-1)) + (b(n) + a(n) + a(n-1)) u(n),
\end{array}$$
$z\in\R$, and $\partial u(n)=u(n+1)-u(n)$ is the usual forward difference operator, is the discrete analog of the Sturm–-Liouville differential equation.

We call $u(z)$ a \emph{solution} of (\ref{eq:jde}) if $(\tau - z) u(z) = 0$ and $u(z)\not\equiv 0$. Whenever the spectral parameter is evident from the context we abbreviate $u=u(z)$. To any two initial values $u(n_0), u(n_0+1), n_0\in\Z$, there exists a unique 'solution' $u$ of (\ref{eq:jde}) which vanishes if and only if $(u(n_0),u(n_0+1))=(0,0)$. We exclude this case and thus, a solution of (\ref{eq:jde}) cannot have two consecutive zeros. We call $n$ a \emph{node} of $u$ if
\be u(n)=0 \quad\text{or}\quad a(n) u(n)u(n+1)>0 \ee
and say that a node $n$ of $u$ lies between $m$ and $l$ if either $m<n<l$ or if $n=m$ and $u(m)\ne 0$. The \emph{number of nodes of $u$ between $m$ and $l$} is denoted as $\#_{(m,l)}(u)$. 

From classical oscillation theory originating in the seminal work of Sturm from 1836 \cite{stu} we know that the $n$-th eigenfunction of a Sturm--Liouville operator has $n-1$ nodes. This also holds for eigensequences of Jacobi operators, cf.\ e.g.\ \cite{fort}, \cite{siosc}, or \cite{oscjac}. Our aim now is to show that the number of nodes of the Wronskian determinant of two (suitable) solutions $u_j(\lam_j)$ of $(\tau_j-\lam_j)u_j=0$, $j=0,1$, equals the difference of the number of eigenvalues of $J_0$ and $J_1$ in $(\lam_0,\lam_1)$.

In \cite{reloscjacm} (cf.\ also \cite{thesis}) Teschl and myself considered the special case $a_0=a_1$ which is now generalized to arbitrary perturbations. We still assume $a_0, a_1<0$. This is no restriction since altering the sign of one or more elements of $a$ does not affect the spectrum of the corresponding matrices, their similarity can easily be shown. Nevertheless, the signs of the solutions of the underlying difference equations depend on the signs of $a$ and therefore we assume $a<0$ to simplify (\ref{eq:countingmethod}).

The Wronskian is given by $W(u_0,u_1)\in\ell(\Z)$, where
\be W_n(u_0,u_1) = u_0(n) a_1(n) u_1(n+1) - u_1(n) a_0(n) u_0(n+1). \ee
We set
\be \label{eq:countingmethod}
\#_n(u_0,u_1) = \left\{\begin{array}{rl}
1	& \text{if } W_{n+1}(u_0,u_1) u_0(n+1) u_1(n+1) > 0 \text{ and } \\
	& \quad\text{either } W_n(u_0,u_1) W_{n+1}(u_0,u_1) < 0 \\ 
	& \quad\text{or } W_n(u_0,u_1) = 0 \text{ and } W_{n+1}(u_0,u_1) \neq 0 \\[2mm]
-1 & \text{if } W_n(u_0,u_1) u_0(n+1) u_1(n+1) > 0 \text{ and } \\
  & \quad\text{either } W_n(u_0,u_1) W_{n+1}(u_0,u_1) < 0 \\
  & \quad\text{or } W_n(u_0,u_1) \neq 0 \text{ and } W_{n+1}(u_0,u_1) = 0 \\[2mm]
0	& \text{otherwise}
\end{array}\right.
\ee
and say the Wonskian has a \emph{(weighted) node} at $n$ if $\#_n(u_0,u_1)\ne 0$. We denote the \emph{number of weighted nodes of the Wronskian between $m$ and $n$}, $m<n$, by
\be \#_{[m,n]}(u_0, u_1) = \sum_{j=m}^{n-1} \#_j(u_0,u_1) \ee
and set
\be\begin{array}{l}
\#_{(m,n]}(u_0, u_1) = \#_{[m,n]}(u_0, u_1) -
\begin{cases}
1&\text{if } W_m(u_0,u_1) = 0 \\
0&\text{otherwise,}
\end{cases} \\
\#_{[m,n)}(u_0, u_1) = \#_{[m,n]}(u_0, u_1) +
\begin{cases}
1&\text{if } W_n(u_0,u_1) = 0 \\
0&\text{otherwise,}
\end{cases} \\
\end{array}\ee
and
\be\begin{array}{l}
\#_{(m,n)}(u_0, u_1) \\
\quad = \#_{[m,n]}(u_0, u_1) -
\begin{cases}
1&\text{if } W_m(u_0,u_1) = 0 \\
0&\text{otherwise}
\end{cases}+\begin{cases}
1&\text{if } W_n(u_0,u_1) = 0 \\
0&\text{otherwise.}
\end{cases}
\end{array}\ee
Here we slightly changed the notation compared to \cite{reloscjacm}: $\#_{(m,n)}$ from \cite{reloscjacm} is now denoted as $\#_{(m,n]}$.

We find the following
\begin{theorem}[Relative Oscillation Theorem]\label{thm:main}
Let $a_0, a_1<0$, then,
\begin{align} \label{eq:main}
&E_{(-\infty,\lam_1)}(J_1) - E_{(-\infty,\lam_0]}(J_0) \nonumber \\
&\quad = \#_{(0,N]}(u_{0,+}(\lam_0),u_{1,-}(\lam_1)) = \#_{(0,N]}(u_{0,-}(\lam_0),u_{1,+}(\lam_1)) \\
\intertext{and}
&E_{(-\infty,\lam_1)}(J_1) - E_{(-\infty,\lam_0)}(J_0) \nonumber \\
&\quad = \#_{[0,N]}(u_{0,+}(\lam_0),u_{1,-}(\lam_1)) = \#_{(0,N)}(u_{0,-}(\lam_0),u_{1,+}(\lam_1)), \nonumber \\
&E_{(-\infty,\lam_1]}(J_1) - E_{(-\infty,\lam_0]}(J_0) \nonumber \\
&\quad = \#_{(0,N)}(u_{0,+}(\lam_0),u_{1,-}(\lam_1)) = \#_{[0,N]}(u_{0,-}(\lam_0),u_{1,+}(\lam_1)), \label{eq:main2} \\
&E_{(-\infty,\lam_1]}(J_1) - E_{(-\infty,\lam_0)}(J_0) \nonumber \\
&\quad = \#_{[0,N)}(u_{0,+}(\lam_0),u_{1,-}(\lam_1)) = \#_{[0,N)}(u_{0,-}(\lam_0),u_{1,+}(\lam_1)), \nonumber
\end{align}
where $E_\Omega(J)$ is the number of eigenvalues of $J$ in $\Omega\subseteq\R$ and $u_{-/+}$ are solutions fulfilling the left/right Dirichlet boundary condition of $J$, i.e.~$u_-(0)=u_+(N)=0$.
\end{theorem}

Equation (\ref{eq:main}) generalizes Theorem~1.2 from \cite{reloscjacm} to different $a$'s. Analogous results for Sturm--Liouville operators were first given by Kr\"uger and Teschl \cite{relosc2,relosc}. For the case of Dirac operators see Stadler and Teschl in \cite{dirac}. For extensions to symplectic eigenvalue problems see Elyseeva \cite{el2010,el2011,el2012}.

In the sequel (Sections \ref{sec:wronski}--\ref{sec:wronskinodes}) we prove Theorem~\ref{thm:main} using the discrete Pr\"ufer transformation. Compared to \cite{reloscjacm,relosc2,relosc,dirac} we also present a simplified proof which eliminates the need to interpolate between operators. This is of particular importance in the present case, since $a_0 < a_1$ does not imply the corresponding relation for the operators, which would make the interpolation step more difficult. In addition, (\ref{eq:main2}) is new. The proofs for regular Sturm--Liouville operators \cite[Theorem~2.3]{relosc2} and regular Dirac operators \cite[Theorem~3.3]{dirac} can be shortened in the same manner and both theorems can be extended to (half-)open and closed spectral intervals analogously to (\ref{eq:main2}) (for the first case cf.~also \cite{ode}).

An extension of Sturm's classical comparison theorem for nodes of solutions to nodes of Wronskians has first been established for Sturm--Liouville operators by Kr\"uger and Teschl in \cite{relosc,relosc2}. In Section~\ref{sec:comp} we show that an analogous comparison theorem holds for Wronskians of solutions of Jacobi difference equations in the case $a_0=a_1$, therefore confer also \cite{thesis}. Moreover, we give Sturm-type comparison theorems for arbitrary perturbations of Jacobi matrices, where, unlike the case of Sturm--Liouville operators, we do not obtain a direct dependence on the coefficients of the operators because $a_0\le a_1$ doesn't imply $J_0\le J_1$.

An extension of Theorem~\ref{thm:main} to Jacobi operators on the half line and on the line is in preparation. This will fill the gap that classical oscillation theory is only applicable below the essential spectrum, while relative oscillation theory works perfectly inside gaps of the essential spectrum. We hope that this will stimulate further research, e.g.~to find new relative oscillation criteria as in the Sturm--Liouville--case, \cite{reloscEPA,PerQuasi}.

We'd be remiss not to mention that several other extensions of relative oscillation theory are thinkable, e.g.~to CMV matrices. Only recently, \v{S}imon Hilscher pointed out in \cite{hilscher} that an extension to the case of Jacobi difference equations with a nonlinear dependence on the spectral parameter would be of particular interest.

\section{The Wronskian} \label{sec:wronski}

\begin{definition} We define the (modified) \emph{Wronskian} (also referred to as Wronski determinant or Casorati determinant) by
\be \begin{array}{rll}
W: 	& \mathbb{D}^2 \times \ell(\Z)^2 	&\to \ell(\Z) \\
		&(\tau_0,\tau_1,\varphi,\psi) 					&\mapsto W^{\tau_0,\tau_1}(\varphi,\psi),
\end{array} \ee
where $\mathbb{D}$ denotes the space of difference equations, such that
\be \begin{array}{rl}
W^{\tau_0,\tau_1}_n(\varphi,\psi)	&= \varphi(n) a_1(n) \psi(n+1) - \psi(n) a_0(n) \varphi(n+1) \\
															&= \begin{vmatrix}\varphi(n) & \psi(n) \\ a_0(n) \varphi(n+1) & a_1(n) \psi(n+1) \end{vmatrix}.
\end{array} \ee
\end{definition}

We abbreviate $\Delta b:=b_0-b_1$, $\Delta a:=a_0-a_1$, and $W_n(\varphi,\psi)=W^{\tau_0,\tau_1}_n(\varphi,\psi)$ whenever the corresponding difference equations are evident from the context. Clearly, if $a_0=a_1$ holds, then $W$ equals the Wronskian from \cite{reloscjacm}. We have
\be\begin{array}{l}
W^{\tau_0,\tau_0}(\varphi,\varphi) \equiv 0, \\
W^{\tau_0,\tau_1}(\varphi,\psi) = - W^{\tau_1,\tau_0}(\psi,\varphi), \\
W^{\tau_0,\tau_1}(c~\varphi,\psi) = W^{\tau_0,\tau_1}(\varphi,c~\psi) = c\ W^{\tau_0,\tau_1}(\varphi,\psi), \\
W^{\tau_0,\tau_1}(\varphi+\ti \varphi,\psi) = W^{\tau_0,\tau_1}(\varphi,\psi) + W^{\tau_0,\tau_1}(\ti \varphi,\psi), \\
W^{\tau_0,\tau_1}(\varphi,\psi+\ti \psi) = W^{\tau_0,\tau_1}(\varphi,\psi) + W^{\tau_0,\tau_1}(\varphi,\ti \psi) \\
\end{array}\ee
for all $c\in \R$ and $\varphi, \ti\varphi, \psi, \ti\psi\in\ell(\Z)$.

\begin{lemma} Green's Formula. We find
\begin{align}
&\sum_{j=n}^m (\varphi (\tau_1 \psi) - \psi (\tau_0 \varphi))(j) = W_m(\varphi,\psi) - W_{n-1}(\varphi,\psi) \\
&\quad - \sum_{j=n-1}^{m-1} \Delta a(j) (\varphi(j+1) \psi(j) + \varphi(j) \psi(j+1)) - \sum_{j=n}^m \Delta b(j) \varphi(j) \psi(j). \nonumber
\end{align}
\end{lemma}
\begin{proof} Just a short calculation.
\end{proof}

\begin{corollary} Let $(\tau_j-z)u_j=0$, then
\begin{align}
&W_m(u_0,u_1) - W_{n-1}(u_0,u_1) \\
&\quad = \sum_{j=n-1}^{m-1} \Delta a(j) (u_0(j+1) u_1(j) + u_0(j) u_1(j+1)) + \sum_{j=n}^m \Delta b(j) u_0(j) u_1(j), \nonumber \\
&W_n(u_0,u_1) - W_{n-1}(u_0,u_1) \label{eq:greenW1} \\
&\quad = \Delta a(n-1) (u_0(n) u_1(n-1) + u_0(n-1) u_1(n)) + \Delta b(n) u_0(n) u_1(n). \nonumber
\end{align}
\end{corollary}

If $u$ and $\ti u$ solve $\tau u = z u$, then $W(u,\ti u)$ is constant (and vanishes iff $u$ and $\ti u$ are linearly dependent).

\section{Pr\"ufer Transformation}

From now on let $u$ be a solution of (\ref{eq:jde}) and let $u_{-/+}$ moreover fulfill the left/right Dirichlet boundary condition of $J$.

\begin{lemma} \cite{jacop}. \label{lem:Jev} The Jacobi matrix $J$ has $N-1$ real and simple eigenvalues. Moreover,
\be z \in \sigma(J) \iff u_-(z,N)=0 \iff u_+(z,0) = 0. \ee
\end{lemma}
\begin{proof} Since $J$ is Hermitian all eigenvalues are real: let $z\in\sig(J), Jv=zv$ and $\norm{v}=1$. Then $z = \langle v,zv \rangle = \langle v,Jv \rangle = \langle Jv,v \rangle = \overline{z}$. It can easily be seen that every eigenvector $u$ corresponding to $z$ fullfills $\tau u=zu$ and $u(0)=0$. Hence, by $W_0(u_-(z),u)=0$, $u_-(z)$ and $u$ are linearly dependent.
\end{proof}

\begin{theorem} \label{thm:numbev} \cite{oscjac}, \cite[Theorem~4.7]{jacop}. For all $\lam\in\R$
\be E_{(-\infty,\lam)} (J) = \#_{(0,N)}(u_-(\lam)) = \#_{(0,N)}(u_+(\lam)) \ee
holds.
\end{theorem}

\begin{lemma} \label{lem:simplezeros} If $u(n)=0$, then $u(n-1) u(n+1) <0$.
\end{lemma}
\begin{proof} Since (\ref{eq:jde}) is a three-term-recursion and $u\not\equiv 0$, all zeros of $u$ are simple and
\be u(n+1) = \underbrace{- a(n)^{-1}}_{>0} (\underbrace{a(n-1)}_{<0} u(n-1) + \underbrace{(b(n) -z) u(n)}_{=0}) \ne 0 \ee
holds.
\end{proof}

By $(u(n),u(n+1)) \ne (0,0)$ for all $n\in\Z$ we can introduce \emph{Pr\"ufer variables}: let $\rho_u, \theta_u \in \ell(\Z)$ denote sequences so that
\be\begin{array}{rcl}\label{eq:pv}
u(n) &=& \rho_u(n) \sin\theta_u(n), \\
- a(n) u(n+1) &=& \rho_u(n) \cos\theta_u(n),
\end{array}\ee
and $\rho_u(n)>0$ holds for all $n\in\Z$. Choose $\theta_u(n_0)\in(-\pi,\pi]$ at the initial position $n_0$ and assume
\be \label{eq:normalth} \ceil{\theta_u(n)/ \pi} \le \ceil{\theta_u(n+1)/ \pi} \le \ceil{\theta_u(n)/ \pi} + 1 \ee
for all $n\in\Z$, then both sequences are well-defined and unique. Here, $x \mapsto \lceil x \rceil = \min \{n \in \Z \,|\, n \geq x \}$ denotes the ceiling function, a left-continuous analog to the well-known floor function $x \mapsto \floor{x} = \max \{n \in \Z \,|\, n \leq x \}$ which itself is a right-continuous step function.

We follow \cite{relosc2} and use the slightly refined (compared to \cite{reloscjacm,jacop,oscjac}) definition of Pr\"ufer variables by taking the secondary diagonals $a$ into account. By $-a>0$ this will not influence the herein recalled claims on the nodes of solutions, but it simplifies our calculations as soon as we look at the nodes of the Wronskian.

\begin{lemma} \label{lem:nodeiff} Fix some $n\in\Z$, then $\exists~k\in\Z$ s.t.~$\theta_u(n) = k \pi + \gamma$ and $\theta_u(n+1) = k \pi + \Gamma$, where
\be\begin{array}{rll}\label{eq:nodeiff}
\gamma \in (0, \frac{\pi}{2}], &\Gamma \in (0,\pi]    &\iff\quad n \text{ is not a node of }u, \\
\gamma \in (\frac{\pi}{2}, \pi], &\Gamma \in (\pi,2\pi) &\iff\quad n \text{ is a node of }u \\
\end{array}\ee
holds. Moreover,
\be \label{eq:nodeiff2} \theta_u(n)=k\pi+\frac{\pi}{2} \quad\iff\quad \theta_u(n+1)=(k+1)\pi. \ee
\end{lemma}
\begin{proof} Abbreviate $\theta=\theta_u$. Choose $k\in\Z$ such that $\theta(n)=k\pi+\gamma$, $\gamma \in (0,\pi]$ holds. By (\ref{eq:normalth}) we have $\Gamma \in (0,2\pi]$.
If $u(n)u(n+1)\ne 0$, then $\sin\gamma\cos\gamma>0$ iff $n$ is not a node of $u$ and $\sin\gamma\cos\gamma<0$ iff $n$ is a node of $u$, hence (\ref{eq:nodeiff}) clearly holds for $\gamma$.
By $\sin\Gamma\cos\gamma>0$ we have $\sin\Gamma>0$ iff $n$ is not a node of $u$ and $\sin\Gamma<0$ iff $n$ is a node of $u$, thus, (\ref{eq:nodeiff}) also holds for $\Gamma$.

Now, suppose we have $u(n+1)=0$, then $n$ is not a node of $u$ and either $\Gamma=\pi$ or $\Gamma=2\pi$ holds. By Lemma~\ref{lem:simplezeros} we have $u(n)u(n+2)<0$, hence $\sin\theta(n) \cos\theta(n+1)=(-1)^k \sin\gamma (-1)^k \cos \Gamma<0$. Thus, by $\cos \Gamma<0$, we have $\Gamma=\pi$. From $-a(n) u(n+1) = \rho(n) \cos \theta(n)=0$ we conclude that $(-1)^k \cos\gamma =0$, thus $\gamma=\frac{\pi}{2}$ and hence (\ref{eq:nodeiff}) and (\ref{eq:nodeiff2}) hold. If $u(n)=0$, then $n$ is a node of $u$, $\gamma=\pi$, and (\ref{eq:nodeiff}) holds by $\sin\theta(n+1)\cos\theta(n)>0$, i.e.~$(-1)^k\sin\Gamma (-1)^k\cos\gamma>0$.
\end{proof}

\begin{corollary} \label{cor:unode} For all $n\in\Z$ we have
\be \ceil{\theta_u(n+1)/\pi} = \begin{cases}
\ceil{\theta_u(n)/\pi} + 1 & \text{if $n$ is a node of $u$} \\
\ceil{\theta_u(n)/\pi} & \text{otherwise.}
\end{cases}\ee
\end{corollary}

Now we are able to count nodes of solutions of the Jacobi difference equation using Pr\"ufer variables and the number of nodes in an interval $(m,n)$ is given by

\begin{theorem} \label{thm:countnodesu} \cite[Lemma~2.5]{oscjac}. We have
\be \#_{(m,n)}(u) = \ceil{\theta_u(n)/\pi} - \floor{\theta_u(m)/\pi} - 1. \ee
\end{theorem}
\begin{proof}
We use mathematical induction: let $n=m+1$, then if $u(m)=0$, $u(n)\neq0$ we have $\#_{(m,n)}(u)=0$ and by Corollary~\ref{cor:unode}
\be \ceil{\theta_u(n)/\pi}=\ceil{\theta_u(m+1)/\pi} = \ceil{\underbrace{\theta_u(m)/\pi}_{\in\Z}}+1=\floor{\theta_u(m)/\pi}+1 \ee
holds. If $u(m)\ne 0$ holds, then by Corollary~\ref{cor:unode} we have
\be \floor{\underbrace{\theta_u(m)/\pi}_{\notin\Z}} = \ceil{\theta_u(m)/\pi} - 1 = \begin{cases}
\ceil{\theta_u(n)/\pi} - 2 & \text{if $m$ is a node} \\
\ceil{\theta_u(n)/\pi} - 1 & \text{otherwise.}
\end{cases}\ee
The inductive step follows again from Corollary~\ref{cor:unode}.
\end{proof}

Let $s_{-/+}(z)$ denote the solution of $\tau s = z s$ fulfilling
\be s_-(z,0)=0, s_-(z,1)=1, \quad\text{resp.~}s_+(z,N)=0, s_+(z,N+1)=1 \ee
and let $n_0$ denote the base point, i.e.~$n_0=0$, resp.~$n_0=N$. Then, by $s_\pm(n_0)=0$ we have $\sin\theta_\pm(n_0)=0$ and by $s_\pm(n_0+1)=1$ we have $-a(n_0)s_\pm(n_0+1)=\rho_s(n_0)\cos\theta_\pm(n_0)>0$, hence $\theta_\pm(n_0)=0$ holds by $\theta_\pm(n_0)\in(-\pi,\pi]$.

\begin{corollary} \label{cor:countnodess} We have
\be \#_{(0,N)}(s_-) = \ceil{\theta_{s_-}(N)/\pi}-1 \quad\text{and}\quad \#_{(0,N)}(s_+) = -\floor{\theta_{s_+}(0)/\pi}-1. \ee
\end{corollary}

\begin{lemma} \label{lem:ev} We find
\be\begin{array}{l}
E_{(-\infty,\lam_1)} (J_1) - E_{(-\infty,\lam_0)} (J_0) \\
\quad= \ceil{\Delta_{s_{0,+}(\lam_0),s_{1,-}(\lam_1)}(N)/\pi} - \ceil{\Delta_{s_{0,+}(\lam_0),s_{1,-}(\lam_1)}(0)/\pi} \\
\quad= \floor{\Delta_{s_{0,-}(\lam_0),s_{1,+}(\lam_1)}(N)/ \pi}-\floor{\Delta_{s_{0,-}(\lam_0),s_{1,+}(\lam_1)}(0)/\pi}, \\
E_{(-\infty,\lam_1)} (J_1) - E_{(-\infty,\lam_0]}(J_0) \\
\quad= \ceil{\Delta_{s_{0,\pm}(\lam_0),s_{1,\mp}(\lam_1)}(N)/\pi} - \floor{\Delta_{s_{0,\pm}(\lam_0),s_{1,\mp}(\lam_1)}(0)/\pi}-1, \\
E_{(-\infty,\lam_1]} (J_1) - E_{(-\infty,\lam_0)} (J_0) \\
\quad= \floor{\Delta_{s_{0,\pm}(\lam_0),s_{1,\mp}(\lam_1)}(N)/\pi} - \ceil{\Delta_{s_{0,\pm}(\lam_0),s_{1,\mp}(\lam_1)}(0)/\pi}+1,\text{ and} \\
E_{(-\infty,\lam_1]} (J_1) - E_{(-\infty,\lam_0]} (J_0) \\
\quad= \ceil{\Delta_{s_{0,-}(\lam_0),s_{1,+}(\lam_1)}(N)/\pi} - \ceil{\Delta_{s_{0,-}(\lam_0),s_{1,+}(\lam_1)}(0)/\pi} \\
\quad= \floor{\Delta_{s_{0,+}(\lam_0),s_{1,-}(\lam_1)}(N)/\pi} - \floor{\Delta_{s_{0,+}(\lam_0),s_{1,-}(\lam_1)}(0)/\pi},
\end{array}\ee
where $\Delta_{u,v}=\theta_v-\theta_u\in\ell(\Z)$.
\end{lemma}
\begin{proof} Abbreviate $s_{j,\pm}=s_{j,\pm}(\lam_j)$. By Theorem~\ref{thm:numbev}, Corollary~\ref{cor:countnodess}, and $-\ceil{x} = \floor{-x}$ for all $x\in\R$ we have
$$\begin{array}{l}
E_{(-\infty,\lam_1)} (J_1) - E_{(-\infty,\lam_0)} (J_0) = \#_{(0,N)}(s_{1,-})-\#_{(0,N)}(s_{0,+}) \\
\quad = \ceil{\theta_{s_{1,-}}(N)/ \pi} - \ceil{-\theta_{s_{0,+}}(0)/ \pi} = \ceil{\Delta_{s_{0,+},s_{1,-}}(N)/\pi} - \ceil{\Delta_{s_{0,+},s_{1,-}}(0)/\pi} \\
\quad = -(E_{(-\infty,\lam_0)} (J_0) - E_{(-\infty,\lam_1)} (J_1)) = \floor{\Delta_{s_{0,-},s_{1,+}}(N)/\pi} - \floor{\Delta_{s_{0,-},s_{1,+}}(0)/\pi}.
\end{array}$$
By Lemma~\ref{lem:Jev} and (\ref{eq:pv}) we have
$$\begin{array}{l}
\lam_0\in\sig(J_0) \quad\iff\quad \Delta_{s_{0,-},s_{1,+}}(N)/\pi\in\Z \quad\iff\quad \Delta_{s_{0,+},s_{1,-}}(0)/\pi\in\Z, \\
\lam_1\in\sig(J_1) \quad\iff\quad \Delta_{s_{0,+},s_{1,-}}(N)/\pi\in\Z \quad\iff\quad \Delta_{s_{0,-},s_{1,+}}(0)/\pi\in\Z
\end{array}$$
and hence
\be E_{(-\infty,\lam_1)} (J_1) - E_{(-\infty,\lam_0]} (J_0) = \ceil{\Delta_{s_{0,\pm},s_{1,\mp}}(N)/ \pi}-\floor{\Delta_{s_{0,\pm},s_{1,\mp}}(0)/ \pi} - 1 \ee
holds by
\be\begin{array}{l}
E_{(-\infty,\lam_1)}(J_1) - E_{(-\infty,\lam_0)}(J_0) \\
\quad= \ceil{\Delta_{s_{0,\pm},s_{1,\mp}}(N)/ \pi}-\floor{\Delta_{s_{0,\pm},s_{1,\mp}}(0)/ \pi} - \begin{cases}
1 &\text{ if $\lam_0\notin\sig(J_0)$} \\
0 &\text{ if $\lam_0\in\sig(J_0)$.}
\end{cases}\end{array}\ee
The rest now follows analogously.
\end{proof}

\section{Nodes of the Wronskian} \label{sec:wronskinodes}

It remains to investigate the sign-changes of $W(u_0,u_1)$. We will express them in terms of Pr\"ufer angles of the involved solutions to finally gain their connection to the difference of the spectra of the corresponding matrices by Lemma~\ref{lem:ev}.

Therefore let $u_j$ be solutions of $\tau_j-z, j=0,1$, where $\rho_j,\theta_j\in\ell(\Z)$ are their Pr\"ufer variables from (\ref{eq:pv}). They correspond to the same spectral parameter $z$, which is no restriction, since we can always replace $b_1$ by $b_1-(z_1-z_0)$. We abbreviate
\be
\Delta:=\Delta_{u_0,u_1} = \theta_1 - \theta_0\in\ell(\Z)
\ee
and adopt Lemma~\ref{lem:nodeiff2} and Lemma~\ref{lem:normaldelta} from \cite{reloscjacm}:
\begin{lemma} \label{lem:nodeiff2} \cite{reloscjacm}. Fix some $n\in\Z$, then $\exists~k_j\in\Z, j=0,1,$ s.t.
\be
\begin{array}{l@{\quad}l}
\theta_j(n) = k_j \pi + \gamma_j, &\quad\gamma_j \in (0,\pi], \\
\theta_j(n + 1) = k_j \pi + \Gamma_j, &\quad\Gamma_j \in (0,2 \pi),
\end{array}\ee
where
\begin{description}
\item [(1)] either $u_0$ and $u_1$ have a node at $n$ or both do not have a node at $n$, then
\be \gamma_1 - \gamma_0 \in (-\frac{\pi}{2}, \frac{\pi}{2}) \quad\text{ and }\quad \Gamma_1 - \Gamma_0 \in (-\pi,\pi). \ee
\item [(2)] $u_1$ has no node at $n$, but $u_0$ has a node at $n$, then
\be \gamma_1 - \gamma_0 \in (-\pi,0) \quad\text{ and }\quad \Gamma_1 - \Gamma_0 \in (-2\pi,0). \ee
\item [(3)] $u_1$ has a node at $n$, but $u_0$ has no node at $n$, then
\be \gamma_1 - \gamma_0 \in (0,\pi) \quad\text{ and }\quad \Gamma_1 - \Gamma_0 \in (0, 2\pi). \ee
\end{description}
\end{lemma}
\begin{proof} Use Lemma~\ref{lem:nodeiff}.
\end{proof}

\begin{lemma} \label{lem:normaldelta} \cite{reloscjacm}. We have
\be \ceil{\Delta(n)/ \pi} -1 \leq \ceil{\Delta(n + 1)/ \pi} \leq \ceil{\Delta(n)/ \pi} + 1. \ee
\end{lemma}
\begin{proof} Let $k:=k_1-k_0$, $n \in \Z$. By Lemma~\ref{lem:nodeiff2} we have either
\be\begin{array}{l@{\quad}l@{\quad}l}
\Delta(n) \in (k \pi - \frac{\pi}{2}, k \pi + \frac{\pi}{2}) & \text{and} & \Delta(n+1) \in (k \pi - \pi, k \pi + \pi), \\
\Delta(n) \in (k \pi - \pi, k \pi) & \text{and} & \Delta(n+1) \in (k \pi - 2\pi, k \pi),\text{ or} \\
\Delta(n) \in (k \pi, k \pi + \pi) & \text{and} & \Delta(n+1) \in (k \pi, k \pi + 2 \pi). \\
\end{array}\ee
In each case the lemma holds.
\end{proof}

\begin{lemma} We have
\begin{align}
\label{eq:Wpv0} W_n(u_0,u_1)&=\rho_0(n) \rho_1(n) \sin \Delta(n), \\
\label{eq:Wpv1} W_n(u_0,u_1) u_0(n+1) u_1(n+1)&=p \sin(\gamma_1-\gamma_0) \cos \gamma_0 \cos \gamma_1, \\
\label{eq:Wpv2} W_{n+1}(u_0,u_1) u_0(n+1) u_1(n+1)&=\ti p \sin(\Gamma_1-\Gamma_0) \cos\gamma_0 \cos\gamma_1,
\end{align}
where $p,\ti p>0$.
\end{lemma}
\begin{proof} Consider
\be \begin{array}{ll}
W_n(u_0,u_1)&=u_0(n) a_1(n) u_1(n+1) - u_1(n) a_0(n) u_0(n+1) \\
&=\rho_0(n) \rho_1(n) \sin(\theta_1(n) - \theta_0(n)) \\
&=\rho_0(n) \rho_1(n) (-1)^{k_1-k_0} \sin(\gamma_1(n) - \gamma_0(n))
\end{array} \ee
and set $p=\frac{\rho_0(n)^2 \rho_1(n)^2}{a_0(n) a_1(n)}$ and $\ti p=\frac{\rho_0(n) \rho_1(n) \rho_0(n+1) \rho_1(n+1)}{a_0(n) a_1(n)}$.
\end{proof}

\begin{lemma} We have
\be\begin{array}{ll}
u_0(n+1)=u_1(n+1)=0 &\implies W_n(u_0,u_1)=W_{n+1}(u_0,u_1)=0, \\
u_0(n+1)=0, u_1(n+1)\ne 0 &\implies W_n(u_0,u_1)W_{n+1}(u_0,u_1)>0, \\
u_0(n+1)\ne 0, u_1(n+1)=0 &\implies W_n(u_0,u_1)W_{n+1}(u_0,u_1)>0.
\end{array}\ee
\end{lemma}
\begin{proof} The first claim holds trivially. For the second claim just observe that by Lemma~\ref{lem:simplezeros}
\be W_n(u_0,u_1) W_{n+1}(u_0,u_1)=-u_0(n)u_0(n+2)a_0(n+1)a_1(n)u_1(n+1)^2>0\ee
holds if $u_0(n+1)=0, u_1(n+1)\ne 0$ and
\be W_n(u_0,u_1) W_{n+1}(u_0,u_1)=-u_1(n)u_1(n+2)a_0(n)a_1(n+1)u_0(n+1)^2>0\ee
holds if $u_0(n+1)\ne 0, u_1(n+1)=0$.
\end{proof}
\begin{corollary}\label{cor:WnodeNoZeros} If $W_n(u_0,u_1)W_{n+1}(u_0,u_1)<0$ or $W_n(u_0,u_1)=0, W_{n+1}(u_0,u_1) \neq 0$ or $W_n(u_0,u_1)\neq 0, W_{n+1}(u_0,u_1)=0$, then
\be u_0(n+1)u_1(n+1)\neq 0 \ee
and moreover $\Delta a(n)\neq 0$ or $\Delta b(n+1)\neq 0$ holds.
\end{corollary}

To shorten notation we denote
\be\begin{array}{cl}
\up &\text{if\quad $\ceil{\Delta(n+1)/\pi}=\ceil{\Delta(n)/\pi}+1$}, \\
\nc &\text{if\quad $\ceil{\Delta(n+1)/\pi}=\ceil{\Delta(n)/\pi}$, and } \\
\down &\text{if\quad $\ceil{\Delta(n+1)/\pi}=\ceil{\Delta(n)/\pi}-1$}.
\end{array}\ee

\begin{lemma} \label{lem:Wnodes} Let $n\in\Z$, then
\begin{align}
\up &\iff W_{n+1}(u_0,u_1) u_0(n+1) u_1(n+1)>0\text{ and} \nonumber \\
& \qquad\qquad\text{either}\quad W_n(u_0,u_1)W_{n+1}(u_0,u_1)<0 \label{eq:Wnodes1} \\
& \qquad\qquad \text{or}\quad W_n(u_0,u_1)=0, W_{n+1}(u_0,u_1)\ne 0, \nonumber \\
\down &\iff W_n(u_0,u_1) u_0(n+1) u_1(n+1)>0\text{ and} \nonumber \\
& \qquad\qquad\text{either}\quad W_n(u_0,u_1)W_{n+1}(u_0,u_1)<0 \label{eq:Wnodes2} \\
& \qquad\qquad\text{or}\quad W_n(u_0,u_1)\ne 0, W_{n+1}(u_0,u_1)=0, \nonumber \\
\nc &\iff \text{otherwise.}
\end{align}
\end{lemma}
\begin{proof} If $\up$, then we either have case (1) of Lemma~\ref{lem:nodeiff2} and $\gamma_1-\gamma_0\in(-\frac{\pi}{2},0], \Gamma_1-\Gamma_0\in(0,\pi)$ or we have case (3) of Lemma~\ref{lem:nodeiff2} and $\gamma_1-\gamma_0\in(0,\pi), \Gamma_1-\Gamma_0\in(\pi,2\pi)$. Clearly, by (\ref{eq:Wpv0}), in either case we have
\be \text{$W_n(u_0,u_1)W_{n+1}(u_0,u_1)<0$ \quad or \quad $W_n(u_0,u_1)=0, W_{n+1}(u_0,u_1)\ne 0$.} \ee
Hence, by Corollary~\ref{cor:WnodeNoZeros} we have $u_0(n+1)u_1(n+1)\ne 0$ and thus $\cos\gamma_0\cos\gamma_1\ne 0$. In case (1) of Lemma~\ref{lem:nodeiff2} we have $\sin(\Gamma_1-\Gamma_0)>0$ and $\cos\gamma_0\cos\gamma_1>0$ by Lemma~\ref{lem:nodeiff}. Hence, by (\ref{eq:Wpv2}) $W_{n+1}(u_0,u_1) u_0(n+1) u_1(n+1)>0$ holds. In case (3) of Lemma~\ref{lem:nodeiff2} we have $\sin(\Gamma_1-\Gamma_0)<0$ and $\cos\gamma_0\cos\gamma_1<0$ by Lemma~\ref{lem:nodeiff}. Hence, by (\ref{eq:Wpv2})
\be W_{n+1}(u_0,u_1) u_0(n+1) u_1(n+1)>0 \ee holds.

If $\down$, then we either have case (1) of Lemma~\ref{lem:nodeiff2} and $\gamma_1-\gamma_0\in(0,\frac{\pi}{2}), \Gamma_1-\Gamma_0\in(-\pi,0]$ or we have case (2) of Lemma~\ref{lem:nodeiff2} and $\gamma_1-\gamma_0\in(-\pi,0), \Gamma_1-\Gamma_0\in(-2\pi,-\pi]$. Clearly, by (\ref{eq:Wpv0}), in either case we have
\be \text{$W_n(u_0,u_1)W_{n+1}(u_0,u_1)<0$ \quad or \quad $W_n(u_0,u_1)\ne 0, W_{n+1}(u_0,u_1)=0$.} \ee
Hence, by Corollary~\ref{cor:WnodeNoZeros} we have $u_0(n+1)u_1(n+1)\ne 0$ and thus $\cos\gamma_0\cos\gamma_1\ne 0$. In case (1) of Lemma~\ref{lem:nodeiff2} we have $\sin(\gamma_1-\gamma_0)>0$ and $\cos\gamma_0\cos\gamma_1>0$ by Lemma~\ref{lem:nodeiff}. Hence, by (\ref{eq:Wpv1}) $W_n(u_0,u_1) u_0(n+1) u_1(n+1)>0$ holds. In case (2) of Lemma~\ref{lem:nodeiff2} we have $\sin(\gamma_1-\gamma_0)<0$ and $\cos\gamma_0\cos\gamma_1<0$ by Lemma~\ref{lem:nodeiff}. Hence, by (\ref{eq:Wpv1})
\be W_n(u_0,u_1) u_0(n+1) u_1(n+1)>0 \ee
holds.

On the other hand, if $W_n(u_0,u_1)W_{n+1}(u_0,u_1)<0$ by (\ref{eq:Wpv0}) we have either $\up$ or $\down$. If, use (\ref{eq:Wpv1}),
\be W_n(u_0,u_1) u_0(n+1) u_1(n+1)=p \sin(\gamma_1-\gamma_0) \cos \gamma_0 \cos \gamma_1>0, \ee
then we have either case (1) or case (2) of Lemma~\ref{lem:nodeiff2} and in each case we have $\nc$ or $\down$. Hence,
$$ \text{$W_n(u_0,u_1)W_{n+1}(u_0,u_1)<0$ and $W_n(u_0,u_1) u_0(n+1) u_1(n+1)>0\implies\down$.} $$
If, use (\ref{eq:Wpv1}),
\be W_n(u_0,u_1) u_0(n+1) u_1(n+1)=p \sin(\gamma_1-\gamma_0) \cos \gamma_0 \cos \gamma_1<0, \ee
then we have either case (1) or case (3) of Lemma~\ref{lem:nodeiff2} and in each case we have $\nc$ or $\up$. Hence,
$$ \text{$W_n(u_0,u_1)W_{n+1}(u_0,u_1)<0$ and $W_{n+1}(u_0,u_1) u_0(n+1) u_1(n+1)>0\implies\up$.} $$

If $W_n(u_0,u_1)=0, W_{n+1}(u_0,u_1)\ne 0$, then we have case (1) of Lemma~\ref{lem:nodeiff2} and by Corollary~\ref{cor:WnodeNoZeros} we have $\cos\gamma_0\cos\gamma_1>0$. Hence, if $W_{n+1}(u_0,u_1) u_0(n+1) u_1(n+1)>0$, then (\ref{eq:Wpv2}) implies $\sin(\Gamma_1-\Gamma_0)>0$, thus, $\up$ holds by case (1) of Lemma~\ref{lem:nodeiff2}.

If $W_n(u_0,u_1)\ne 0, W_{n+1}(u_0,u_1)=0$, then by Corollary~\ref{cor:WnodeNoZeros} we have $\cos\gamma_0\cos\gamma_1\ne 0$. If additionally $W_n(u_0,u_1) u_0(n+1) u_1(n+1)>0$ holds, then by (\ref{eq:Wpv1}) $\cos\gamma_0\cos\gamma_1$ and $\sin(\gamma_1-\gamma_0)$ are of the same sign. Hence, we have case (1) of Lemma~\ref{lem:nodeiff2} and $\down$ or case (2) of Lemma~\ref{lem:nodeiff2} and $\down$.

Thus, (\ref{eq:Wnodes1}) and (\ref{eq:Wnodes2}) hold and clearly by Lemma~\ref{lem:normaldelta} we have $\nc$ otherwise.
\end{proof}

\begin{remark}
Consider (\ref{eq:countingmethod}), then
\be\begin{array}{l}
\text{$W_n(u_0,u_1) W_{n+1}(u_0,u_1) \neq 0$ or $W_n(u_0,u_1)=W_{n+1}(u_0,u_1)=0$} \\
\quad\implies\#_n(u_0,u_1)=-\#_n(u_1,u_0), \\
W_n(u_0,u_1)W_{n+1}(u_0,u_1)<0\implies\#_n(u_0,u_1)\ne 0
\end{array}\ee
by Corollary~\ref{cor:WnodeNoZeros}.
Moreover, if $W_n(u_0,u_1)=0$ and $W_{n+1}(u_0,u_1)\neq0$ holds, then $u_0(n)=0 \iff u_1(n)=0$. 
\end{remark}

That (\ref{eq:countingmethod}) is a generalization of the counting method established in \cite[(1.8)]{reloscjacm}, where $\Delta a=0$ holds, follows from (\ref{eq:greenW1}). From Lemma~\ref{lem:Wnodes} we conclude
\begin{align}
\#_n(u_0,u_1) = \ceil{\Delta(n+1)/\pi} - \ceil{\Delta(n)/\pi}, \\
\#_{[m,n]}(u_0,u_1) = \ceil{\Delta(n)/\pi} - \ceil{\Delta(m)/\pi}. \label{eq:WnodeDelta}
\end{align}

\begin{lemma} \label{lem:WnodeDelta} We have
\begin{align}
\#_{(m,n]}(u_0, u_1) &= \ceil{\Delta(n)/\pi} - \floor{\Delta(m)/\pi}-1, \\
\#_{[m,n)}(u_0, u_1) &= \floor{\Delta(n)/\pi} - \ceil{\Delta(m)/\pi}+1,\text{ and} \\
\#_{(m,n)}(u_0, u_1) &= \floor{\Delta(n)/\pi} - \floor{\Delta(m)/\pi}.
\end{align}
\end{lemma}
\begin{proof} By (\ref{eq:Wpv0}) we have $W_j(u_0,u_1)=0\iff\Delta(j)/\pi\in\Z$ and hence by (\ref{eq:WnodeDelta})
\be\begin{array}{rl}
\#_{(m,n]}(u_0, u_1) &= \ceil{\Delta(n)/\pi} - \ceil{\Delta(m)/\pi} -
\begin{cases}
0&\text{if $W_m(u_0,u_1) \neq 0$} \\
1&\text{if $W_m(u_0,u_1) = 0$}
\end{cases} \\
&= \ceil{\Delta(n)/\pi} - \floor{\Delta(m)/\pi} - 1
\end{array}\ee
holds. The second and the third claim follow analogously.
\end{proof}

\begin{lemma} We have
\be \#_{[m,n]}(u_0,u_1) = - \#_{(m,n)}(u_1,u_0), \quad \#_{(m,n]}(u_0,u_1) = - \#_{[m,n)}(u_1,u_0). \ee
If $W_m(u_0,u_1)\ne 0$ and $W_n(u_0,u_1)\ne 0$, then $\#_{[m,n]}(u_0,u_1)=-\#_{[m,n]}(u_1,u_0)$.
\end{lemma}
\begin{proof} Use $\ceil{x}=-\floor{-x}$ and Lemma~\ref{lem:WnodeDelta}.
\end{proof}

\begin{proof} [Proof of Theorem~\ref{thm:main}] By Lemma~\ref{lem:ev} and Lemma~\ref{lem:WnodeDelta} we have
\be\begin{array}{l}
E_{(-\infty,\lam_1)}(J_1) - E_{(-\infty,\lam_0]}(J_0) \\
\quad= \ceil{\Delta_{s_{0,\pm}(\lam_0),s_{1,\mp}(\lam_1)}(N)/\pi} - \floor{\Delta_{s_{0,\pm}(\lam_0),s_{1,\mp}(\lam_1)}(0)/\pi}-1 \\
\quad = \#_{(0,N]}(s_{0,\pm}(\lam_0),s_{1,\mp}(\lam_1)) = \#_{(0,N]}(u_{0,\pm}(\lam_0),u_{1,\mp}(\lam_1)).
\end{array}\ee
Equations (\ref{eq:main2}) can be shown analogously.
\end{proof}

\begin{remark} By Theorem~\ref{thm:main} we have
\be \#_{[0,N]}(u_{0,\pm}(\lam),u_{1,\mp}(\lam_1)) = - \#_{[0,N]}(u_{1,\pm}(\lam),u_{0,\mp}(\lam)), \ee
$$\begin{array}{l}
\#_{[0,N]}(u_{0,+}(\lam),u_{3,-}(\lam)) \\
\quad = \#_{[0,N)}(u_{0,+}(\lam),u_{1,-}(\lam)) + \#_{[0,N]}(u_{1,-}(\lam),u_{2,+}(\lam)) + \#_{(0,N]}(u_{2,+}(\lam),u_{3,-}(\lam)),
\end{array}$$
and
$$\begin{array}{l}
\#_{[0,N]}(u_{0,-}(\lam),u_{3,+}(\lam)) \\
\quad = \#_{(0,N]}(u_{0,-}(\lam),u_{1,+}(\lam)) + \#_{[0,N]}(u_{1,+}(\lam),u_{2,-}(\lam)) + \#_{[0,N)}(u_{2,-}(\lam),u_{3,+}(\lam)).
\end{array}$$
\end{remark}

\section{Triangle Inequality and Comparison Theorem} \label{sec:comp}

In this section we establish the Triangle Inequality and the Comparison Theorem for Wronskians which generalize Theorem 5.12 and Theorem 5.13 from \cite{thesis} to different $a$'s. Moreover, Theorem~\ref{thm:nodesWu} generalizes and sharpens Theorem 5.11 from \cite{thesis}.

\begin{theorem} [Comparison Theorem for Wronskians \RM{1}] Let $J_1\ge J_2$, then,
\be \#_{[0,N]}(u_{0,\pm}(\lam),u_{2,\mp}(\lam)) \ge \#_{[0,N]}(u_{0,\pm}(\lam),u_{1,\mp}(\lam)), \ee
where $\#_{[0,N]}$ can be replaced by $\#_{(0,N]}$, $\#_{[0,N)}$, or $\#_{(0,N)}$.
\end{theorem}
\begin{proof} Let $\sig(J_1)=\{\lam_1, \dots, \lam_{N-1}\}$ and $\sig(J_2)=\{\ti\lam_1, \dots, \ti\lam_{N-1}\}$, then $\lam_i\ge \ti\lam_i$ for all $i$ by $J_1\ge J_2$, cf.~\cite[Theorem~8.7.1]{lancaster}, and hence we have $E_{(-\infty,\lam)}(J_2) \ge E_{(-\infty,\lam)}(J_1)$. Thus, by Theorem~\ref{thm:main}
\be\begin{array}{l}
\#_{[0,N]}(u_{0,+}(\lam),u_{2,-}(\lam)) = E_{(-\infty,\lam)}(J_2) - E_{(-\infty,\lam)}(J_0) \\
\quad \ge E_{(-\infty,\lam)}(J_1) - E_{(-\infty,\lam)}(J_0) = \#_{[0,N]}(u_{0,+}(\lam),u_{1,-}(\lam)).
\end{array}\ee
The other claims follow analogously from $E_{(-\infty,\lam]}(J_2) \ge E_{(-\infty,\lam]}(J_1)$ and Theorem~\ref{thm:main}.
\end{proof}

\begin{corollary} Let $a_0=a_1=a_2$ and $b_0(j)\ge b_1(j)\ge b_2(j)$ for all $j=1, \dots, N-1$. If $0$ and $N-1$ are positive nodes of $W(u_{0,\pm}(\lam),u_{1,\mp}(\lam))$, then $W(u_{0,\pm}(\lam),u_{2,\mp}(\lam))$ has at least two positive nodes at $0, \dots, N-1$.
\end{corollary}

\begin{theorem} \label{thm:nodesWu} Let $m<n$, then
\be \abs{\#_{[m,n]}(u_0,u_1) - (\#_{(m,n)}(u_1) - \#_{(m,n)}(u_0))} \leq 1, \ee
where $\#_{[m,n]}$ can be replaced by $\#_{(m,n]}$ or $\#_{[m,n)}$.
\end{theorem}
\begin{proof} For all $x,y\in\R$ we have
\be 0 \leq \ceil{x-y} - (\ceil{x} - \ceil{y}) \leq  1 \quad\text{and}\quad -1\leq \floor{x-y} - (\floor{x} - \floor{y}) \leq 0. \ee
Hence, by (\ref{eq:WnodeDelta}), Theorem~\ref{thm:countnodesu}, and $-\ceil{x}=\floor{-x}$ we have
$$\begin{array}{l}
\abs{\#_{[m,n]}(u_0, u_1) - (\#_{(m,n)}(u_1) - \#_{(m,n)}(u_0))} \\
\quad =\abs{\ceil{\frac{\theta_1(n)-\theta_0(n)}{\pi}} - (\ceil{\frac{\theta_1(n)}{\pi}} - \ceil{\frac{\theta_0(n)}{\pi}}) + \floor{\frac{\theta_0(m)-\theta_1(m)}{\pi}} - (\floor{\frac{\theta_0(m)}{\pi}} - \floor{\frac{\theta_1(m)}{\pi}})} \\
\quad \le 1.
\end{array}$$
Moreover, by Lemma~\ref{lem:WnodeDelta} and Theorem~\ref{thm:countnodesu} we have
\be\begin{array}{l}
\#_{(m,n]}(u_0, u_1) - (\#_{(m,n)}(u_1) - \#_{(m,n)}(u_0)) \\
\quad =\ceil{\frac{\Delta(n)}{\pi}} - (\ceil{\frac{\theta_1(n)}{\pi}} - \ceil{\frac{\theta_0(n)}{\pi}}) - (\floor{\frac{\Delta(m)}{\pi}} - (\floor{\frac{\theta_1(m)}{\pi}} - \floor{\frac{\theta_0(m)}{\pi}})) - 1
\end{array}\ee
and
$$\begin{array}{l}
\#_{[m,n)}(u_0, u_1) - (\#_{(m,n)}(u_1) - \#_{(m,n)}(u_0)) \\
= 1 + \floor{\frac{\theta_0(m)-\theta_1(m)}{\pi}} - (\floor{\frac{\theta_0(m)}{\pi}} - \floor{\frac{\theta_1(m)}{\pi}}) - (\ceil{\frac{\theta_0(n)-\theta_1(n)}{\pi}} - (\ceil{\frac{\theta_0(n)}{\pi}} - \ceil{\frac{\theta_1(n)}{\pi}})).
\end{array}$$
\end{proof}

\begin{theorem}[Triangle Inequality for Wronskians]\label{thm:triangle} We have
\be \abs{\#_{[m,n]}(u_0,u_2) - (\#_{[m,n]}(u_0,u_1) + \#_{[m,n]}(u_1,u_2))} \leq 1, \ee
where $\#_{[m,n]}$ can be replaced by $\#_{(m,n]}$.
\end{theorem}
\begin{proof} Abbreviate $\Delta_{i,j}:=\Delta_{u_i,u_j}$, then $\Delta_{0,1} + \Delta_{1,2} = \Delta_{0,2}$. By (\ref{eq:WnodeDelta}) we have $\#_{[m,n]}(u_0,u_2) = \ceil{\Delta_{0,2}(n)/\pi} - \ceil{\Delta_{0,2}(m)/\pi}$, hence
\be\begin{array}{l}
\#_{[m,n]}(u_0,u_1) + \#_{[m,n]}(u_1,u_2) \\
\quad \leq \ceil{\Delta_{0,2}(n)/\pi}+1 - \ceil{\Delta_{0,2}(m)/\pi} = \#_{[m,n]}(u_0,u_2) + 1, \\
\#_{[m,n]}(u_0,u_1) + \#_{[m,n]}(u_1,u_2) \\
\quad \geq \ceil{\Delta_{0,2}(n)/\pi} - (\ceil{\Delta_{0,2}(m)/\pi}+1) = \#_{[m,n]}(u_0,u_2) - 1
\end{array}\ee
holds by $\ceil{x+y} \leq \ceil{x}+\ceil{y}\leq \ceil{x+y}+1$ for all $x,y\in \R$. Further, by Lemma~\ref{lem:WnodeDelta} and $\floor{x+y}-1\le\floor{x} + \floor{y}\le\floor{x+y}$ we have
\be\begin{array}{l}
\#_{(m,n]}(u_0,u_1) + \#_{(m,n]}(u_1,u_2) \\
\quad \le \ceil{\Delta_{0,2}(n)/\pi} - \floor{\Delta_{0,2}(m)/\pi} = \#_{(m,n]}(u_0,u_2) + 1 \\
\end{array}\ee
and $\#_{(m,n]}(u_0,u_2) \leq \#_{(m,n]}(u_0,u_1) + \#_{(m,n]}(u_1,u_2) + 1$.
\end{proof}

\begin{theorem}[Comparison Theorem for Wronskians \RM{2}] If either
\begin{description}
\item [A] $W_j(u_0,u_1) u_0(j+1) u_1(j+1)\le 0$ and $W_j(u_1,u_2) u_1(j+1) u_2(j+1)\le 0$ \\
\quad for all $j=0, \dots, N-2$ or
\item [B] $a_0=a_1=a_2$ and $b_0(j)\ge b_1(j)\ge b_2(j)$ for all $j=1, \dots N-1$
\end{description}
holds and $0$ and $N-2$ are (positive) nodes of $W(u_0,u_1)$, then $W(u_0,u_2)$ has at least one positive node at $0, \dots, N-2$.
\end{theorem}
\begin{proof} In either case we have $\#_j(u_0,u_1)\ge 0$ and $\#_j(u_1,u_2)\ge 0$ for all $j=0, \dots, N-2$ and hence from Theorem~\ref{thm:triangle} we conclude
\be \#_{[0,N-1]}(u_0,u_2)\ge \underbrace{\#_{[0,N-1]}(u_0,u_1)}_{\ge 2}+\underbrace{\#_{[0,N-1]}(u_1,u_2)}_{\ge 0}-1. \ee
\end{proof}

\noindent
{\bf Acknowledgments.}
I wish to thank Gerald Teschl for several useful discussions.

\end{document}